\documentclass[a4paper, 12pt,reqno]{amsart}
\usepackage{amsmath,amssymb,amsthm,enumerate}
\allowdisplaybreaks
\theoremstyle{plain}
\newtheorem{theorem}{Theorem}
\newtheorem*{theorem*}{Theorem}
\newtheorem{lemma}{Lemma}
\newtheorem*{lemma*}{Lemma}
\theoremstyle{definition}

\newtheorem*{definition*}{Definition}
\theoremstyle{remark}
\newtheorem{remark}{Remark}

\newtheorem*{remark*}{Remark}
\newtheorem{conjecture}{Conjecture}
\newtheorem*{statement*}{Statement}

\begin{document}
\title[Generalizations of the Salem function]{One modification of the Salem function}

\author{Symon Serbenyuk}



\maketitle
\text{\emph{simon6@ukr.net}}\\
\text{\emph{Kharkiv National University of Internal Affairs, Ukraine }}
\begin{abstract}

The present article is devoted to one example which related to the Salem function. The main attention is given to properties of one type of functions including items related to functional equations, graphs, the Lebesgue integral, etc.

\end{abstract}

\section{Introduction}

Fractal sets \cite{sets2, sets1, sets} and  functions with complicated local structure and are pathological objects in real analysis. A class of hese functions contains singular (for example, \cite{{Minkowski}, {Salem1943}, {S.Serbenyuk 2017}, {Zamfirescu1981}}),  nowhere monotonic \cite{Symon2017, Symon2019}, and nowhere differentiable functions  (for example, see \cite{{Bush1952}, {Serbenyuk-2016}}, etc.).  An interest in such functions can be explained by their  connection with modelling  real objects, processes, and phenomena (in physics, economics, technology, etc.) and with different areas of mathematics (for example, see~\cite{ACFS2011, BK2000, Kruppel2009, OSS1995, Symon2021,   Symon21, Symon21-1, Sumi2009, Takayasu1984, TAS1993}). For modeling such functions various numeral systems and operators (auxiliary functions) are used (\cite{Renyi1957, S.Serbenyuk, preprint1-2018, preprint2019}, etc.).

Researchers are trying to find simpler examples of singular  functions. For example (see \cite{S. Serbenyuk systemy rivnyan 2-2} and references therein), in 1830, the first example of a continuous  non-differentiable  function was modeled by Bolzano in  ``Doctrine on Function" but the last paper was published one   hundred years later.  Brief historical remarks on  functions with complicated local structure are given in~\cite{{ACFS2017}, {S. Serbenyuk systemy rivnyan 2-2}}.

One of the simplest examples of singular functions was introduced by Salem in \cite{Salem1943}. The Salem  function is a function of the following form:
$$
S(x)=S\left(\Delta^q _{i_1i_2...i_k...}\right)=\beta_{i_1}+ \sum^{\infty} _{k=2} {\left(\beta_{i_k}\prod^{k-1} _{r=1}{p_{i_r}}\right)}=y=\Delta^{P_q} _{i_1i_2...i_k...},
$$
where $q>1$ is a fixed positive integer, $p_j>0$ for all $j=\overline{0,q-1}$, and $p_0+p_1\dots + p_{q-1}=1$. Here
$$
\Delta^q _{i_1i_2...i_k...}:=\sum^{\infty} _{k=1}{\frac{i_k}{q^k}}, ~~~i_k\in\{0, 1, 2, 3, \dots , q-1 \}.
$$
That is, any value (an arbitrary number from $[0,1]$) of the Salem function can be represented by the number notation $\Delta^{P_q} _{i_1i_2...i_k...}$ for a fixed positive integer $q>1$. The last representation is called \emph{the $P_q$-representation of $x$}. In this paper, the main attention is given to a function whose arguments  represented by the $P_3$-representation.

One can remark that generalizations of the Salem function can be singular, non-differentiable functions,  or those that do not have a derivative on a certain set. There are many researches which are devoted to the Salem function and its generalizations or modifications in terms of various representations of an argument (for example, see \cite{ACFS2017, Kawamura2010, Symon2015, Symon2017, Symon2019, Symon2021, Symon2023} and references in these papers). 

 In 2012,  the author of this article presented at  the International Scientific Conference ``Asymptotic Methods in the Theory of Differential Equations" dedicated to 80th anniversary of M.~I.~Shkil~\cite{S. Serbenyuk abstract 6} and published in the paper\footnote{ The working paper available at   https://www.researchgate.net/publication/314409844, as well as the version of the last-mentioned published paper into English available at https://arxiv.org/pdf/1703.02820.pdf} \cite{Symon12(2)} the following function:
\begin{equation}
\label{ff1}
x=\Delta^{3} _{i_{1}i_{2}...i_{k}...}\stackrel{f}{\rightarrow} \Delta^{3} _{\theta(i_{1})\theta(i_{2})...\theta(i_{k})...}=f(x)=y,
\end{equation}
where $\Delta^{3} _{i_{1}i_{2}...i_{n}...}$ is the ternary  representation of  $x \in [0,1]$ and values of this function are obtained from the ternary representation of the argument by the following change of digits: 0 by 0, 1 by 2, and 2 by 1. That is,  this function preserves the ternary digit $0$. Also, numbers, whose  ternary representation has the period $(2)$ (without the number $1$), are not used under the consideration of $f$.  

\begin{remark}
\label{Remark1} 
In 2012, in  the mentioned papers \cite{{S. Serbenyuk abstract 6}, {Symon12(2)}}, it is noted that one can define $m=3!=6$ functions  determined on $[0,1]$ in terms of the ternary numeral system by the following way:
$$
\Delta^{3} _{i_{1}i_{2}...i_{k}...}\stackrel{f_m}{\rightarrow} \Delta^{3} _{\theta_m(i_{1})\theta_m(i_{2})...\theta_m(i_{k})...},
$$
where the function $\theta_m(i_k)$ determined on an alphabet of the ternary numeral system and $f_m (x)$ is defined by the following table for each $m=\overline{1,6}$. 
\begin{center}
\begin{tabular}{|c|c|c|c|}
\hline
 &$ $ 0 &$ 1 $ & $2$\\
\hline
$\theta_1 (i_k) $ &$0$ & $1$ & $2$\\
\hline
$\theta_2 (i_k) $ &$0$ & $2$ & $1$\\
\hline
$\theta_3 (i_k) $ &$1$ & $0$ & $2$\\
\hline
$\theta_4 (i_k) $ &$1$ & $2$ & $0$\\
\hline
$\theta_5 (i_k) $ &$2$ & $0$ & $1$\\
\hline
$\theta_6 (i_k) $ &$2$ & $1$ & $0$\\
\hline
\end{tabular}
\end{center}
That is, one can model a class of functions whose values  are obtained from the ternary representation
of the argument by  a certain  change of ternary digits. It is easy to see that the function $f_1 (x)$ is the function  $y=x$ and the function $f_6 (x)$ is the function $y=1-x$, i.e.,
$$
y=f_1(x)=f_1\left(\Delta^3 _{i_1i_2\ldots i_k\ldots}\right)=\Delta^3 _{i_1i_2\ldots i_k \ldots}=x,
$$
$$
y=f_6(x)=f_6\left(\Delta^3 _{i_1i_2\ldots i_k\ldots}\right)=\Delta^3 _{[2-i_1][2-i_2]\ldots  [2-i_k]\ldots}=1-x.
$$

Generalizations of such functions for the $q$- and nega-$q$- representations of arguments, are investigated in \cite{{Serbenyuk-2016}}. 

For the function $y=f_2(x)$, which argument and values are defined by  the ternaryy numeral system, one can note the following properties (\cite{{S. Serbenyuk abstract 6}, {Symon12(2)}, {S. Serbenyuk systemy rivnyan 2-2}}). The function  $f_2$:
\begin{itemize}
\item  is  not monotonic on the domain of definition;

\item  is continuous at ternary-irrational points, and ternary-rational points are points of discontinuity of the function;

\item  is non-differentiable;

\item has self-similar properties and its Lebesgue integral is equal to $\frac{1}{2}$;

\item   satisfies the following functional equation:
\begin{equation*}
f(x)-f(1-x)=x-\frac{1}{2}.
\end{equation*}
\end{itemize}
\end{remark}

The present paper is devoted to a function of the type of form \eqref{ff1} (in the other words, to the function $f_2$), which argument and values are represented by values of the Salem function for the case of $q=3$.

The paper is organized as follows. In Section~\ref{object}, the considered function is defined. Also, the main attention is given to its well-posedness, the contunuity, and certain map properties such as the set of invariant points, the bijectivity, and the monotonicity, etc., as well as to modelling the function by systems of functional equations, to the self-affinity of the graph, and to integral properties. In addition, the Section~\ref{derivative} is devoted to differential properies of the function. 

Finally, one can remark certain peculiarities of the present investigation which are that the present function has more complicated definition and techniques for proving its  properties than in the case of the ternary representation, as well as there are unknown relations of   a numeral system (in terms  of which are defined arguments and values  of the function) and some properties of  our function in comparison with the same properties of function~\eqref{ff1}, which is defined in terms of the other expansion of real numbers. Final remarks will be given in conclusions.

\section{The main object}
\label{object}

Let us consider a function of the form
\begin{equation}
\label{ff2}
x=\Delta^{P_3} _{i_{1}i_{2}...i_{k}...}\stackrel{f}{\rightarrow} \Delta^{P_3} _{\theta(i_{1})\theta(i_{2})...\theta(i_{k})...}=\beta_{\theta(i_1)}+ \sum^{\infty} _{k=2} {\left(\beta_{\theta(i_k)}\prod^{k-1} _{r=1}{p_{\theta(i_r)}}\right)}=f(x)=y,
\end{equation}
where $\theta(0)=0$, $\theta(1)=2$, and $\theta(2)=1$, as well as
$$
x= \Delta^{P_3} _{i_1i_2...i_k...}\equiv \beta_{i_1}+ \sum^{\infty} _{k=2} {\left(\beta_{i_k}\prod^{k-1} _{r=1}{p_{i_r}}\right)}.
$$
Here $i_k \in\{0, 1, 2\}$ and $p_{i_k}\in (0, 1)$ for any $k\in\mathbb N$, as well as $p_0+p_1+p_2=1$.

\begin{remark}
Numbers of the form
$$
\Delta^{P_3} _{i_{1}i_{2}...i_{m-1}i_m000...}:=\Delta^{P_3} _{i_{1}i_{2}...i_{m-1}i_m(0)}=\Delta^{P_3} _{i_{1}i_{2}...i_{m-1}[i_m-1](2)}:=\Delta^{P_3} _{i_{1}i_{2}...i_{m-1}[i_m-1]222...},
$$
where $i_m\ne 0$, are called \emph{$P_3$-rational}. That is, $P_3$-rational are the following numbers:
$$
x^{'} _1=\Delta^{P_3} _{i_{1}i_{2}...i_{m-1}1(0)}=\Delta^{P_3} _{i_{1}i_{2}...i_{m-1}0(2)}=x^{''} _1,
$$
$$
x^{'} _2=\Delta^{P_3} _{i_{1}i_{2}...i_{m-1}2(0)}=\Delta^{P_3} _{i_{1}i_{2}...i_{m-1}1(2)}=x^{''} _2.
$$
\end{remark}

\begin{lemma}
Values of the function $f$ for different representations of $P_3$-rational numbers are different.
\end{lemma}
\begin{proof}
It is easy to see that
$$
f\left(x^{'} _1\right)=\Delta^{P_3} _{\theta(i_{1})\theta(i_{2})...\theta(i_{m-1})2(0)}\ne \Delta^{P_3} _{\theta(i_{1})\theta(i_{2})...\theta(i_{m-1})0(1)}=f\left(x^{''} _1\right)
$$
and
$$
f\left(x^{'} _2\right)=\Delta^{P_3} _{\theta(i_{1})\theta(i_{2})...\theta(i_{m-1})1(0)}\ne \Delta^{P_3} _{\theta(i_{1})\theta(i_{2})...\theta(i_{m-1})2(1)}=f\left(x^{''} _2\right).
$$
\end{proof}
\begin{remark}
\label{Remark3}
To reach that the function $f$ be well-defined on the set of $P_3$-rational numbers from $[0, 1]$, we will not consider the $P_3$-representations, which have period $(2)$ (without the number $1$).
\end{remark}

\begin{lemma}
The function $f$ has the following properties:
\begin{enumerate}
\item $f$ maps the closed interval $[0,1]$ into $[0,1]$ without a certain enumerable subset, i.e.,
$$
f: [0,1]\stackrel{f}{\rightarrow} [0,1] \setminus \left\{y: y=\Delta^{P_3} _{j_1j_2...j_k111...},~i_k\ne 1\right\}.
$$
\item $f$ has the unique invariant point and the equality $f(0)=0$ holds.
\item the function $f$ is not bijective on the domain.
\item $f$ is not a monotonic function on the domain.
\end{enumerate}
\end{lemma}
\begin{proof}
\emph{The first property} follows from Remark~\ref{Remark3}. \emph{The second property} follows from the definition of $f$.

Let us prove \emph{the third property}. Suppose $x_1=\Delta^{P_3} _{i_{1}i_{2}...i_{k}...}$ and $x_2=\Delta^{P_3} _{j_{1}j_{2}...j_{k}...}$ are $P_3$-irrational and $x_1\ne x_2$.

Let us  find the following set
$$
\{x: f(x_1)=f(x_2), x_1 \ne x_2\}.
$$

If $y_0=f(x_1)=f(x_2)$ is $P_3$-irrational, then the following must be hold:
$$
y_0=\Delta^{P_3} _{\alpha_{1}\alpha_{2}...\alpha_{k}...}=\Delta^{P_3} _{\theta(i_{1})\theta(i_{2})...\theta(i_{k})...}=\Delta^{P_3} _{\theta(j_{1})\theta(j_{2})...\theta(j_{k})...}.
$$
Since the last equalities hold and $y_0$ is $P_3$-irrational, we have $x_1=x_2$ but this contradicts the condition $x_1\ne x_2$.

Suppose $y_0$ is $P_3$-irrational. Then there exists $k_0$ such that the condition $\theta(i_t)=\theta(j_t)$ holds for all $t=\overline{1,k_0}$, as well as
$$
y_{1,2}=\Delta^{P_3} _{\theta(i_1)\theta(i_2)...\theta(i_{k_0})\theta(i_{k_0+1})(0)}=\Delta^{P_3} _{\theta(j_1)\theta(\j_2)...\varphi(\beta_{n_0})(\theta(i_{k_0+1})-1)(2)}
$$
or
$$
y_{1,2}=\Delta^{P_3} _{\theta(j_1)\theta(j_2)...\theta(j_{k_0})\theta(j_{k_0+1})(0)}=\Delta^{P_3} _{\theta(i_1)\theta(i_2)...\theta(i_{k_0})(\theta(j_{k_0+1})-1)(2)}.
$$
That is,
$$
\left[
\begin{aligned}
\left\{
\begin{aligned}
\theta(i_{k_0+2})=\theta(i_{k_0+3})=...&=0\\
\theta(j_{k_0+2})=\theta(j_{k_0+3})=...&=2\\
\theta(j_{k_0+1})& = \theta(i_{k_0+1})-1\\
\end{aligned}
\right.\\
\left\{
\begin{aligned}
\theta(j_{k_0+2})=\theta(j_{k_0+3})=...&=0\\
\theta(i_{k_0+2})=\theta(i_{k_0+3})=...&=2\\
\theta(i_{k_0+1})& =\theta(j_{k_0+1}) -1.\\
\end{aligned}
\right.
\end{aligned}
\right.
$$

Whence,
$$
\left[
\begin{aligned}
\left\{
\begin{aligned}
i_{k_0+2}=i_{k_0+3}=...&=0\\
j_{k_0+2}=j_{k_0+3}=... & = 1\\
\left[
\begin{aligned}
\left\{
\begin{aligned}
i_{k_0+1}&=2\\
j_{k_0+1} & = 0\\
\end{aligned}
\right.\\
\left\{
\begin{aligned}
i_{k_0+1}&=1\\
j_{k_0+1} & = 2,\\
\end{aligned}
\right.\\
\end{aligned}
\right.\\
\end{aligned}
\right.\\
\left\{
\begin{aligned}
i_{k_0+2}=i_{k_0+3}=...&=1\\
j_{k_0+2}=j_{k_0+3}=... & = 0\\
\left[
\begin{aligned}
\left\{
\begin{aligned}
i_{k_0+1}&=0\\
j_{k_0+1} & = 2\\
\end{aligned}
\right.\\
\left\{
\begin{aligned}
i_{k_0+1}&=2\\
j_{k_0+1} & = 1.\\
\end{aligned}
\right.\\
\end{aligned}
\right.\\
\end{aligned}
\right.\\
\end{aligned}
\right.
$$

So, $f(x_1)=f(x_2)$ holds under the condition $x_1 \ne x_2$ on the following sets:
$$
 G _1=\{x: x_1=\Delta^{P_3} _{c_1c_2...c_{k_0}2(0)} \wedge x_2=\Delta^{P_3} _{c_1c_2...c_{k_0}0(1)}\}
$$
and
$$
 G _2=\{x: x_1=\Delta^{P_3} _{c_1c_2...c_{k_0}1(0)} \wedge x_2=\Delta^{P_3} _{c_1c_2...c_{k_0}2(1)}\},
$$
where $c_1, c_2,..., c_n$ are fixed digits and   $k_0 \in \mathbb Z_0=\mathbb N \cup \{0\}$. The set
$$
G=G_1 \cup G_2 
$$
is  enumerable. 

\emph{The fourth property} follows from the following. Suppose $x_1<x_2$, as well as  $x_1=\Delta^{P_3} _{i_{1}i_{2}...i_{k}...}$ and $x_2=\Delta^{P_3} _{j_{1}j_{2}...j_{k}...}$, i.e., there exists $k_0$ such that $i_r=j_r$ for all $r=\overline{1,k_0}$ and $i_{k_0+1}<j_{k_0+1}$. Considering $f(x_1)$ and$f(x_2)$, one can note the following: the equality $\theta(i_r)=\theta(j_r)$ holds for all $r=\overline{1,k_0}$ but there are $i_{k_0+1}$ and $j_{k_0+1}$ from $\{0,1,2\}$ such that $\theta(i_{k_0+1})<\theta(j_{k_0+1})$ or $\theta(i_{k_0+1})<\theta(j_{k_0+1})$ as well. For example, 
$$
i_{k_0+1}=1<2=j_{k_0+1},  ~~~~~~~~\theta(i_{k_0+1})=2>1=\theta(j_{k_0+1}),
$$
$$
i_{k_0+1}=0<2=j_{k_0+1},  ~~~~~~~~\theta(i_{k_0+1})=0<1=\theta(j_{k_0+1}).
$$
So, $f$ is not monotonic. 
\end{proof}

\begin{lemma}
The function   $f$ is continuous at $P_3$-irrational  points, and the $P_3$-rational  points are points of discontinuity of the function.
\end{lemma}
\begin{proof}
Let $x_0=\Delta^{P_3} _{i_1i_2...i_k...}$ be an arbitrary $P_3$-irrational number from $[0,1]$.
 Let $x=\Delta^{P_3} _{j_1j_2...j_k...}$ be a  $P_3$-irrational number such that the condition $i_{r}=j_{r}$ holds  for all $j=\overline{1,k_0}$ and $i_{k_0+1}\ne j_{k_0+1}$, where $k_0$ is a certain positive integer. 
Since $f$ is a bounded function, $0\le f(x) \le 1$, we obtain $g(x)-g(x_0)=$
$$
=\beta_{\theta(i_1)}+\sum^{m} _{k=2}{\left(\beta_{\theta(i_k)}\prod^{k-1} _{t=1}{p_{\theta{(i_t)}}}\right)}+\prod^{k_0} _{u=1}{p_{\theta(i_u)}}\left(\beta_{\theta(i_{k_0+1})}+\sum^{\infty} _{s=k_0+2}{\left(\beta_{\theta(i_s)}\prod^{s-1} _{l=k_0+1}{p_{\theta(i_l)}}\right)}\right)
$$
$$
-\beta_{\theta(i_1)}-\sum^{m} _{k=2}{\left(\beta_{\theta(i_k)}\prod^{k-1} _{t=1}{p_{\theta{(i_t)}}}\right)}-\prod^{k_0} _{u=1}{p_{\theta(i_u)}}\left(\beta_{\theta(j_{k_0+1})}+\sum^{\infty} _{s=k_0+2}{\left(\beta_{\theta(j_s)}\prod^{s-1} _{l=k_0+1}{p_{\theta(j_l)}}\right)}\right)
$$
$$
=\prod^{k_0} _{u=1}{p_{\theta(i_u)}}\left(\beta_{\theta(i_{k_0+1})}+\sum^{\infty} _{s=k_0+2}{\left(\beta_{\theta(i_s)}\prod^{s-1} _{l=k_0+1}{p_{\theta(i_l)}}\right)}-\beta_{\theta(j_{k_0+1})}-\sum^{\infty} _{s=k_0+2}{\left(\beta_{\theta(j_s)}\prod^{s-1} _{l=k_0+1}{p_{\theta(j_l)}}\right)}\right)
$$
$$
\le (1-0)\prod^{k_0} _{u=1}{p_{\theta(i_u)}}=\prod^{k_0} _{u=1}{p_{\theta(i_u)}}.
$$
Hence
$$
\lim_{k_0\to \infty}{|f(x)-f(x_0)|}=\lim_{k_0\to \infty}{\prod^{k_0} _{u=1}{p_{\theta(i_u)}}}\le\lim_{k_0\to \infty}{\left(\max\{p_0, p_1, p_2\}\right)^{k_0}}=0.
$$

So, $\lim_{x\to x_0}{f(x)}=f(x_0)$, i.e., the function $f$ is continuous at any $P_3$-irrational point. 

Let $x_0$ be a $P_3$-rational number, i.e.,
$$
x_0=\Delta^{P_3} _{i_1i_2...i_{m-1}i_m(0)}=\Delta^{P_3} _{i_1i_2...i_{m-1}[i_m-1](2)},~~~i_m \ne 0.
$$
Then 
$$
\lim_{x \to x_0 -0} {f(x)}=\Delta^{P_3} _{\theta(i_1)\theta(i_2)...\theta(i_{m-1})\theta(i_m-1)(1)},
$$
$$
\lim_{x \to x_0 +0} {f(x)}=\Delta^{P_3} _{\theta(i_1)\theta(i_2)...\theta(i_{m-1})\theta(i_m)(0)}.
$$
So, $x_0$ is a point of discontinuity.
\end{proof}

Let us describe an auxiliary map $\sigma^n(x)$ which is called the shift operator and is a piecewise linear function. That is, 
$$
\sigma^n(x)=\sigma^n \left(\Delta^{P_3} _{i_1i_2...i_m...}\right)=\sigma^n \left(\Delta^{P_3} _{i_{n+1}i_{n+2}i_{n+3}...}\right)=\beta_{i_{n+1}}+\sum^{\infty} _{s=n+1}{\left(\beta_{i_s} \prod^{s-1} _{l=n+1}{p_{i_l}}\right)}.
$$

\begin{theorem}
Let $P_3=(0, 1, 2)$ be a fixed tuple of real numbers such that $p_t\in (0,1)$, where $t=0, 1, 2$, $\sum_t {p_t}=1$, and $0=\beta_0<\beta_t=\sum^{t-1} _{l=0}{p_l}<1$ for all $t\ne 0$. Then the following system of functional equations
\begin{equation}
\label{eq: system}
f\left(\sigma^{n-1}(x)\right)=\beta_{\theta(i_{n})}+p_{\theta(i_{n})}f\left(\sigma^n(x)\right),
\end{equation}
where $x=\Delta^{P_3} _{i_1i_2...i_k...}$, $n=1,2, \dots$, and $\sigma_0(x)=x$, has the unique solution
$$
f(x)=\beta_{\theta({i_1})}+\sum^{\infty} _{k=2}{\left(\beta_{\theta(i_{k})}\prod^{k-1} _{r=1}{p_{\theta(i_{r})}}\right)}
$$
in the class of determined and bounded on $[0, 1]$ functions. 
\end{theorem}
\begin{proof} Since the function $f$ is a determined and bounded  function on $[0,1]$, using system~\eqref{eq: system},  we get 
$$
f(x)=\beta_{\theta ({i_1})}+p_{\theta({i_1})}f(\sigma(x))=\beta_{\theta({i_1})}+p_{\theta({i_1})}(\beta_{\theta({i_2})}+p_{\theta({i_2})}f(\sigma^2(x)))=\dots
$$
$$
\dots =\beta_{\theta({i_1})}+\beta_{\theta({i_2})}p_{\theta({i_1})}+\beta_{\theta({i_3})}p_{\theta({i_1})}p_{\theta({i_2})}+\dots +\beta_{\theta{i_n}}\prod^{n-1} _{l=1}{p_{\theta({i_l})}}+\left(\prod^{n} _{r=1}{p_{\theta({i_r})}}\right)f(\sigma^n(x)).
$$
Since 
$$
\prod^{n} _{r=1}{p_{\theta(i_r)}}\le \left( \max\{p_0, p_1, p_2\}\right)^n\to 0, ~~~ n\to \infty,
$$
and
$$
\lim_{n\to\infty}{f(\sigma^n(x))\prod^{n} _{r=1}{p_{\theta({i_r})}}}=0,
$$
we have
$$
g(x)=\beta_{\theta(i_1)}+\sum^{\infty} _{k=2}{\left(\beta_{\theta({i_k})}\prod^{k-1} _{r=1}{p_{\theta({i_r})}}\right)}.
$$
\end{proof}

\begin{theorem}
Suppose
$$
\psi_t: \left\{
\begin{array}{rcl}
x^{'}&=&p_tx+\beta_t\\
y^{'} & = &p_{\theta(t)}y+\beta_{\theta(t)}\\
\end{array}
\right.
$$
are affine transformations for $t=0, 1, 2$ and $p_0, p_1, p_2 \in (0, 1)$. Then the graph $\Gamma$ of  $f$ is a self-affine set of $\mathbb R^2$, as well as
$$
\Gamma= \bigcup^{2} _{t=0}{\psi_t(\Gamma)}.
$$
\end{theorem}
\begin{proof}
If $T(x_0, y_0)\in M\subset \Gamma$, then  $x_0=p_tx+\beta_t$ and $y_0=p_{\theta(t)} y+\beta_{\theta(t)}$ for some $t\in \{0, 1, 2\}$. Using the system of functional equations~\eqref{eq: system}, we get
$$
f(x_0)=\beta_{\theta(t)}+p_{\theta(t)}f(x)=y_0
$$ 
and $T\in \Gamma$.

Choose  $T(x_0, f(x_0))\in \Gamma$. Then $x_0=\beta_t+p_t \sigma(x_0)$, $f(x_0)=\beta_{\theta(t)}+p_{\theta(t)}f(\sigma(x_0))$, and $(\sigma(x_0), f(\sigma(x_0)))\in \Gamma$.

So, $\psi_t(\sigma(x_0), f(\sigma(x_0)))=(x_0, f(x_0))\in M$.
\end{proof}

\begin{theorem}
For the Lebesgue integral, the following equality holds:
$$
\int_{[0, 1]}{f(x)dx}=\frac{p^2 _1+p_0p_1+p_0p_2}{1-p^2 _0-2p_1p_2}.
$$
\end{theorem}
\begin{proof}
 Let us begin with some equalities which are useful for the future calculations:
$$
x=\beta_{i_1}+p_{i_1}\sigma(x)
$$
and
$$
dx=p_{i_1}d(\sigma(x)),
$$
as well as
$$
d(\sigma^{n-1} (x))=p_{i_n}d(\sigma^{n} (x))
$$
for all $n=1, 2, 3, \dots$

Let us calculate the Lebesgue integral
$$
I:=\int^1 _0{f(x)dx}=\sum^{2} _{t=0}{\int^{\beta_{t+1}} _{\beta_t}{f(x)dx}}=\sum^{2} _{t=0}{\int^{\beta_{t+1}} _{\beta_t}{(\beta_{\theta(t)}+p_{\theta(t)}f(\sigma(x))})dx}
$$
$$
=\beta_1p_2+\beta_2p_1+\sum^{2} _{t=0}{\int^{\beta_{t+1}} _{\beta_t}{p_{\theta(t)}f(\sigma(x))dx}}
$$
$$
=\beta_1p_2+\beta_2p_1+\sum^{2} _{t=0}{p_{\theta(t)}\int^{\beta_{t+1}} _{\beta_t}{p_tf(\sigma(x))d(\sigma(x))}}
$$
$$
=\beta_1p_2+\beta_2p_1+\sum^{2} _{t=0}{p_{\theta(t)}p_t\int^{\beta_{t+1}} _{\beta_t}{f(\sigma(x))d(\sigma(x))}}.
$$

Using self-affine properties of $f$, we get
$$
I=\beta_1p_2+\beta_2p_1+I\sum^{2} _{t=0}{p_{\theta(t)}p_t}.
$$

So,
$$
I=\frac{\beta_1p_2+\beta_2p_1}{1-p^2 _0-2p_1p_2}=\frac{p^2 _1+p_0p_1+p_0p_2}{1-p^2 _0-2p_1p_2}.
$$
\end{proof}

\section{Differential properties}
\label{derivative}

The function   $f$ is continuous at $P_3$-irrational  points, and the $P_3$-rational  points are points of discontinuity of the function.This section is devoted to the singularity of our function.

\begin{conjecture}
The function $f$ is a singular function at any $P_3$-irrational point.
\end{conjecture}

Let $c_1,c_2,\dots, c_m$ be a fixed 
ordered tuple of integers such that $c_r\in\{0,1,2\}$ for $r=\overline{1,m}$. 

\emph{A cylinder $\Lambda^{P_3} _{c_1c_2...c_m}$ of rank $m$ with base $c_1c_2\ldots c_m$} is the following set 
$$
\Lambda^{P_3} _{c_1c_2...c_m}\equiv\{x: x=\Delta^{P_3} _{c_1c_2...c_m i_{m+1}i_{m+2}\ldots i_{m+k}\ldots}\}.
$$
It is easy to see that  any cylinder $\Delta^Q _{c_1c_2...c_m}$ is a closed interval of the form
$$
\left[\Delta^{P_3} _{c_1c_2...c_m(0)}, \Delta^Q _{c_1c_2...c_m(2)}\right]=\left[\beta_{c_1}+\sum^{m} _{k=2}{\left(\beta_{c_r}\prod^{k-1} _{l=1}{p_{c_l}}\right)}, \beta_{c_1}+\sum^{m} _{k=2}{\left(\beta_{c_r}\prod^{k-1} _{l=1}{p_{c_l}}\right)}+\prod^{m} _{r=1}{p_{c_r}} \right].
$$
Whence,
$$
\left|\Lambda^{P_3} _{c_1c_2...c_m}\right|=\sup \Lambda^{P_3} _{c_1c_2...c_m}-\inf \Lambda^{P_3} _{c_1c_2...c_m}=\prod^{m} _{r=1}{p_{c_r}},
$$
where $|\cdot|$ is the Lebesgue measure of a set.
Then a value of the increment $\mu_f (\cdot)$ of the fucnction $f$ on a set can be calculated as following:
$$
\mu_f\left(\Lambda^{P_3} _{c_1c_2...c_m}\right)= f\left(\sup\Lambda^{P_3} _{c_1c_2...c_m}\right)- f\left(\inf \Lambda^{P_3} _{c_1c_2...c_m}\right)=\Delta^{P_3} _{\theta(c_1)\theta(c_2)...\theta(c_m)(1)}-\Delta^{P_3} _{\theta(c_1)\theta(c_2)...\theta(c_m)(0)}
$$
$$
=\prod^{m} _{r=1}{p_{\theta(c_r)}}\left(\beta_1+\beta_1p_1+\beta_1p^2 _1+\dots \right)=\frac{\beta_1}{1-p_1}\prod^{m} _{r=1}{p_{\theta(c_r)}}.
$$
Hence,
$$
\lim_{m\to\infty}{\frac{\mu_f\left(\Lambda^{P_3} _{c_1c_2...c_m}\right)}{\left|\Lambda^{P_3} _{c_1c_2...c_m}\right|}}=\frac{\beta_1}{1-p_1}\lim_{m\to\infty}{\frac{\prod^{m} _{r=1}{p_{\theta(c_r)}}}{\prod^{m} _{r=1}{p_{c_r}}}}=\frac{\beta_1}{1-p_1}\lim_{m\to\infty}{\prod^{m} _{r=1}{\frac{p_{\theta(c_r)}}{p_{c_r}}}}.
$$

Let us  discuss the other aproach. Choose $x_0=\Delta^{P_3} _{i_1i_2...i_{n_0-1}ii_{n_0+1}i_{n_0+2}i_{n_0+3}...}$, where $i_{n_0}=i$ is a fixed digit; then let us  consider  a sequence $(x_n)$ such that  $x_n=\Delta^{P_3} _{i_1i_2...i_{n_0-1}ji_{n_0+1}i_{n_0+2}i_{n_0+3}...}$.

Suppose $n_0=1, 2, 3, \dots$; then
$$
x_n-x_0=\left(\prod^{n_0-1} _{r=1}{p_{i_r}}\right)\times
$$
$$
\times\left(\beta_j -\beta_i+p_j\left(\beta_{i_{n_0+1}}+\sum^{\infty} _{k=n_0+2}{\left(\beta_{i_k}\prod^{k-1} _{t=n_0+1}{p_{i_t}}\right)}\right)-p_i\left(\beta_{i_{n_0+1}}+\sum^{\infty} _{k=n_0+2}{\left(\beta_{i_k}\prod^{k-1} _{t=n_0+1}{p_{i_t}}\right)}\right)\right)
$$
$$
=\left(\prod^{n_0-1} _{r=1}{p_{i_r}}\right)\left(\beta_j -\beta_i+(p_j-p_i)\left(\beta_{i_{n_0+1}}+\sum^{\infty} _{k=n_0+2}{\left(\beta_{i_k}\prod^{k-1} _{t=n_0+1}{p_{i_t}}\right)}\right)\right)
$$
$$
=\left(\prod^{n_0-1} _{r=1}{p_{i_r}}\right)\left(\beta_j -\beta_i+(p_j-p_i)\sigma^{n_0}(x_n)\right),
$$
$$
f(x_n)-f(x_0)=\left(\prod^{n_0-1} _{r=1}{\theta(p_{i_r})}\right)\left(\beta_{\theta(j)} -\beta_{\theta(i)}+(p_{\theta(j)}-p_{\theta(i)})\sigma^{n_0}(f(x_n))\right),
$$
and
$$
\lim_{x_n-x_0\to 0}{\frac{f(x_n)-f(x_0)}{x_n-x_0}}=\lim_{n_0\to\infty}{\frac{f(x_n)-f(x_0)}{x_n-x_0}}
$$
$$
=\lim_{n_0\to\infty}{\frac{\beta_{\theta(j)} -\beta_{\theta(i)}+(p_{\theta(j)}-p_{\theta(i)})\sigma^{n_0}(f(x_n))}{\beta_j -\beta_i+(p_j-p_i)\sigma^{n_0}(x_n)}\prod^{n_0-1} _{r=1}{\frac{p_{\theta(c_r)}}{p_{c_r}}}}
$$

In the first and second cases, we have a certain number or a bounded sequece which multiplied on 
$$
\prod^{k} _{r=1}{\frac{p_{\theta(c_r)}}{p_{c_r}}},
$$
where $k\to\infty$. Let us evaluate this. 

Suppose $s=0, 1, 2$ and  $N_s(x, k)$ is the number of the digit $s$ in the $k$ first digits of the $P_3$-representation of $x$. Also, 
$$
\lim_{k\to\infty}{\frac{N_s(x,k)}{k}}=\nu_s(x)
$$
is the frequency of the digit $s$ in the $P_3$-representation of $x\in [0,1]$.

So,
$$
\prod^{k} _{r=1}{\frac{p_{\theta(c_r)}}{p_{c_r}}}=\frac{p^{N_0(x, k)} _0 p^{N_2(x, k)} _1 p^{N_1(x, k)} _2}{p^{N_0(x, k)} _0 p^{N_1(x, k)} _1 p^{N_2(x, k)} _2}=\frac{p^{N_2(x, k)} _1 p^{N_1(x, k)} _2}{ p^{N_1(x, k)} _1 p^{N_2(x, k)} _2}=p^{N_2(x, k)-N_1(x, k)} _1 p^{N_1(x, k)-N_2(x, k)} _2
$$

Using (for example \cite{{Salem1943}}, etc.) Salem's techniques (and the fact that the Salem function is continuous and strictly increasing for positive $p_0, p_1, p_2$),  as well as (see \cite{{HW1979}, {PVB2011}} and references therein) a statement  that the set of normal numbers (numbers such that $\nu_s =p_s$ holds for all digit $s$ ) has the full Lebesgue measure, we get
$$
\lim_{k\to\infty}{\prod^{k} _{r=1}{\frac{p_{\theta(c_r)}}{p_{c_r}}}}=0.
$$

So, the derivative is equal to $0$ almost everywhere.

\section{Conclusions}

In the present article, a certain function $f$, which related to the classical Salem function $S$ and to some none-differentiable function $f_2$ (see Remark~\ref{Remark1}) for the case of the ternary represenation of numbers from $[0, 1]$, is investigated.

Finally, one can remark the following peculiarities in properties. Since  in terms of the ternary representation the digit operator $\theta_2$ generates a non-differentiable function, this operator preserves the singularity in terms of representations of numbers by the Salem  function. Such properties as the continuity, the monotonicity,  and  map properties are analogous.   Our function has other self-affine and self-similar properties, as well as  is a solution of a system of functional equations. In addition, positive  parameters $p_0+p_1+p_2=1$ of the Salem function affects on coefficients of the self-similiarity for $f$, as well as on the value of its Lebesgue integral. The last is equal to the value $\frac 1 2$ of the Lebesgue integral of $f_2$ whenever the condition $p_0=p_1=p_2=\frac 1 3$ holds.

\begin{center}
{\bf{Statements and Declarations}}

{{Competing Interests}}

\emph{The author states that there is no conflict of interest}
\end{center}


\begin{thebibliography}{9}

\bibitem{ACFS2011}
de Amo, E., Carrillo, M.D., and  Fern\'andez-S\'anchez, J. On duality of aggregation
operators and k-negations. \emph{Fuzzy Sets and Systems}, \textbf{181}, 14--27 (2011).

\bibitem{ACFS2017}
de Amo, E., Carrillo, M.D., and  Fern\'andez-S\'anchez, J. A Salem generalised function. \emph{Acta Math. Hungar.} \textbf{151}, 361---378 (2017). https://doi.org/10.1007/s10474-017-0690-x


\bibitem{BK2000}
Berg, L.  and  Kruppel, M. De Rham's singular function and related
functions. \emph{Z. Anal. Anwendungen.}, \textbf{19}, no.~1,  227--237 (2000).


\bibitem{Bush1952} Bush K.A., Continuous functions without derivatives.  \emph{Amer. Math. Monthly} \textbf{59},  222--225 (1952). 



\bibitem{HW1979} 
 Hardy, G.H. and  Wright, E.M. \emph{An Introduction to the Theory of Numbers}, 5th ed., Oxford University
Press, Oxford, 1979.



    
\bibitem{Kawamura2010}
Kawamura, K. The derivative of Lebesgue's singular function. \emph{ Real Analysis Exchange}
Summer Symposium 2010, 83--85.

\bibitem{Kruppel2009}
Kruppel, M. De Rham's singular function, its partial derivatives with
respect to the parameter and binary digital sums. \emph{Rostock. Math. Kolloq.}
\textbf{64}, 57--74 (2009).


\bibitem{Minkowski}
 Minkowski, H.  Zur Geometrie der Zahlen. In: Minkowski, H. (ed.) Gesammeine Abhandlungen, Band 2, pp. 50--51. Druck und Verlag von B. G. Teubner, Leipzig und Berlin (1911)


\bibitem{OSS1995}
 Okada, T., Sekiguchi,T.,  and Shiota,  Y. An explicit formula of the
exponential sums of digital sums. \emph{Japan J. Indust. Appl. Math.} \textbf{ 12},
425--438 (1995).

\bibitem{PVB2011}  Parad\'is, J., Viader, P.,  Bibiloni, Ll. A New Singular Function. \emph{ Amer. Math. Monthly}, \textbf{ 118}, no. 4, 344-354 (2011) , DOI: 10.4169/amer.math.monthly.118.04.344


\bibitem{Renyi1957}
 R\'enyi, A. Representations for real numbers and their ergodic properties. \emph{ Acta. Math. Acad. Sci. Hungar.} \textbf{8},  477--493 (1957).



\bibitem{Salem1943}Salem, R. 
 {On some singular monotonic functions which are stricly increasing.}
 {\itshape Trans. Amer. Math. Soc.} {\bf 53}, 423--439 (1943).

\bibitem{S. Serbenyuk abstract 6}  Serbenyuk, S.O. 
 {On one nearly everywhere continuous and  nowhere differentiable function defined by automaton with finite memory.}
 {\itshape International Scientific Conference ``Asymptotic Methods in the Theory of Differential Equations" dedicated to 80th anniversary of M.~I.~Shkil:}  Abstracts, Kyiv: National Pedagogical Dragomanov University, 2012, P. 93 (Ukrainian), available at  https://www.researchgate.net/publication/311665377

\bibitem{Symon12(2)} Serbenyuk, S.O.
 {On one nearly everywhere
continuous and nowhere differentiable function, that defined by automaton with 
finite memory.}
 {\itshape Naukovyi Chasopys NPU im. M. P. Dragomanova. Ser. 1. Phizyko-matematychni Nauky  [Trans. Natl. Pedagog. Mykhailo Dragomanov Univ. Ser. 1.
Phys. Math.]} {\bf 13(2)} (2012). (Ukrainian),  available at https://www.researchgate.net/publication/292970012

 
\bibitem{Symon2015}  Serbenyuk, S.O.
 {Functions, that defined by functional equations systems in terms of Cantor series representation of numbers.}
 {\itshape Naukovi Zapysky NaUKMA} {\bf 165 }, 34--40 (2015). (Ukrainian),  available at https://www.researchgate.net/publication/292606546

\bibitem{Symon2017} Serbenyuk, S.O.
 {Continuous Functions with Complicated Local
Structure Defined in Terms of Alternating Cantor
Series Representation of Numbers.}
 {\itshape  Zh. Mat. Fiz. Anal. Geom.} {\bf 13}, No. 1,  57--81 (2017). 
https://doi.org/10.15407/mag13.01.057



\bibitem{Serbenyuk-2016}
Serbenyuk, S. On one class of functions with complicated local structure. \emph{\v{S}iauliai
Mathematical Seminar} \textbf {11 (19)}, 75--88 (2016).

\bibitem{S. Serbenyuk alternating Cantor series 2013} Serbenyuk, S.  
Representation of real numbers by the alternating Cantor series. {\itshape Integers} {\bf  17},  Paper No.~A15, 27 pp. (2017)
 

\bibitem{S.Serbenyuk 2017} Serbenyuk, S. On one fractal property of the Minkowski function. \emph{Revista de la Real Academia de Ciencias Exactas, F\' isicas y Naturales. Serie A. Matem\' aticas} \textbf{112}, no.~2, 555--559 (2018), doi:10.1007/s13398-017-0396-5

\bibitem{S. Serbenyuk systemy rivnyan 2-2}{  Serbenyuk, S.O.}
Non-Differentiable functions defined in~terms of~classical representations of~real numbers.
 \textit{Zh. Mat. Fiz. Anal. Geom.} \textbf{14}, no.~2,
 197--213 (2018). https://doi.org/10.15407/mag14.02.197




\bibitem{S.Serbenyuk}{Serbenyuk, S.} On some generalizations of real numbers representations, arXiv:1602.07929v1 (in Ukrainian)

\bibitem{preprint1-2018}
Serbenyuk, S. Generalizations of certain representations of real numbers, \emph{Tatra Mountains Mathematical Publications} \textbf{77},  59--72 (2020). https://doi.org/10.2478/tmmp-2020-0033, 
arXiv:1801.10540.

\bibitem{Symon2019} Serbenyuk, S. On one application of infinite systems of functional equations in function theory, \emph{ Tatra Mountains Mathematical Publications} \textbf{74}, 117-144 (2019). https://doi.org/10.2478/tmmp-2019-0024  

\bibitem{preprint2019} Serbenyuk, S. Generalized shift operator of certain encodings of real numbers, arXiv:1911.12140v1, 6 pp. 

\bibitem{preprint19} Serbenyuk, S. On certain functions and related problems,  arXiv:1909.03163 

\bibitem{Symon2021} Serbenyuk, S. Systems of functional equations and generalizations of certain functions. \emph{Aequationes  Mathematicae}   \textbf{95}, 801-820 (2021). https://doi.org/10.1007/s00010-021-00840-8 

\bibitem{Symon2023} Serbenyuk, S. One class of functions with arguments in negative bases, arXiv:2303.07867


\bibitem{Symon21} Serbenyuk, S. Certain functions defined in terms of Cantor series. \emph{ Zh. Mat. Fiz. Anal. Geom.} { \bf{16}} (2020), no. 2, 174-189. https://doi.org/10.15407/mag16.02.174 

\bibitem{Symon21-1} Serbenyuk, S. On certain maps defined by infinite sums. \emph{The Journal of Analysis} 28 (2020), 987-1007. https://doi.org/10.1007/s41478-020-00229-x

\bibitem{sets1} Serbenyuk, S. Some Fractal Properties of Sets Having the Moran Structure. \emph{Tatra Mountains Mathematical Publications} \textbf{81}, no.1, 3922, pp.1-38. https://doi.org/10.2478/tmmp-2022-0001

\bibitem{sets2}  Serbenyuk, S. Certain Singular Distributions and Fractals. \emph{Tatra Mountains Mathematical Publications} \textbf{79}, no.2, 3921, pp.163-198. https://doi.org/10.2478/tmmp-2021-0026

\bibitem{sets} Serbenyuk, S.O. One distribution function on the Moran sets. \emph{ Azerb. J. Math.} \textbf{10}, no.2, 12--30 (2020), arXiv:1808.00395.


\bibitem{Sumi2009}
Sumi, H. Rational semigroups, random complex dynamics and singular
functions on the complex plane. \emph{ Sugaku}  \textbf{61}, no.~2,  133--161 (2009).

\bibitem{Takayasu1984}
 Takayasu, H. Physical models of fractal functions. \emph{Japan J. Appl. Math.}
\textbf{1}, 201--205 (1984).

\bibitem{TAS1993}
Tasaki, S.,Antoniou, I., and Suchanecki,  Z. Deterministic diffusion,
De Rham equation and fractal eigenvectors, \emph{Physics Letter A} \textbf{179}, no.~1,  97--102 (1993).

\bibitem{Zamfirescu1981}Zamfirescu,  T. Most monotone functions are singular. \emph{Amer. Math. Mon.} \textbf{88}, 47--49 (1981).
  

\end{thebibliography}
\end{document}